\documentclass[11pt]{article}

\usepackage{comment}
\usepackage{geometry}                % See geometry.pdf to learn the layout options. There are lots.
\pdfoutput=1
\usepackage{graphicx}
\usepackage{hyperref}
\usepackage{amssymb}
\usepackage{amsmath}
\usepackage{amsthm}
\usepackage{amsfonts}
\usepackage[export]{adjustbox}
\usepackage{tikz}
\usetikzlibrary{calc, shapes, backgrounds}
\usepackage{caption}
\usepackage{subcaption}
\newcommand{\BE}{\begin{equation}}

\newcommand{\EE}{\end{equation}}

\newtheorem{theorem}{Theorem}
\newtheorem{definition}{Definition}
\newtheorem{lemma}{Lemma}

\usepackage[T1]{fontenc}
\title{Counting LEGO configurations.}
\author{Alexander Gunning\\
\small Melbourne, Australia\\
\small\tt rednaxela.backwards@gmail.com\\
Anthony J Guttmann\\
\small School of Mathematics and Statistics,\\
\small The University of Melbourne\\
\small Vic. 3010, Australia\\
\small\tt tony.guttmann@gmail.com\\ and\\
Rasmus M Nilsson\\
\small Copenhagen, Denmark\\
\small\tt fudling@gmail.com}

\begin{document}

\maketitle
\begin{abstract}
We discuss the problem of counting certain LEGO structures, primarily those comprising parallel $w \times 1$ tiles. These can be combined, as a  single LEGO structure, by interlocking the tiles. %Alternatively, if the interlocking condition is relaxed, so that tiles can also be placed end-to-end, a greater number of possible configurations results. 
We also study the historically earlier problem of counting the number of ways to combine $2 \times 4$ LEGO tiles, which in this case gives a 3-dimensional structure.

In all cases the number of configurations is dominated by an exponential growth term, $\mu^n$ where $n$ is the number of tiles. We present an algorithm for counting these various LEGO configurations, and use the data to estimate the asymptotics.

We analyse the data so generated, and conjecture that, for the two-dimensional structures, the number of possible configurations grows like $A(w)\mu(w)^n/n,$ and we give numerical estimates for $A(w)$ and $\mu(w)$  for $w < 11,$ while for the three-dimensional structure the number of possible configurations is conjectured to grow like $A\mu^n/n^{3/2},$ where $\mu = 117.25 \pm 0.05.$

We also study the sequences that arise when we fix the number of tiles $n,$ and vary the tile size $w.$ We prove that the sequences are polynomials of degree $n-1,$ and we give these explicitly for $n=1 \ldots 14.$

\end{abstract}

%\noindent {\bf PACS}: 05.50.+q, 05.10.-a, 02.60.Gf

%\noindent 
%{\bf MSC}: 05A15,  30B10, 82B20, 82B27, 82B41

%\noindent
%{\bf Keywords:} Gerrymander sequence, exact enumeration algorithms,  power-series expansions, asymptotic series analysis  

\section{Introduction}
One of the oldest LEGO counting problems goes back to a LEGO company newsletter of 1974 \cite{C74},
in which it was stated that the number of ways to combine six $2 \times 4$ LEGO tiles of the same colour is 102981500.
This number is not correct however, as it only counts the number of ways to build a structure of height 6, whereas other structures of lesser height can also be constructed from 6 tiles. 

In 2005 Durhuus and Eilers \cite{DE05} reported the correct number, being 915103765\footnote{This result was also independently obtained by Mikkel Abrahamsen.}. They showed, further, that the number of such objects grows exponentially like $\mu^n$ with the number of blocks $n$, and that the growth constant $78.32 \le \mu \le 176.58.$ Note that the structures being counted do not require the tiles to be parallel. Rather, they are counted up to translation and rotation, and restricted to having all blocks oriented north-south or east-west. That is to say, while it is possible to construct interlocking objects with blocks not necessarily parallel or perpendicular to one another, these are excluded. The first 10 coefficients are given in the OEIS \cite{OEIS} as sequence A112389. 

In 2011 Abrahamsen and Eilers \cite{AE11} considered this problem more generally, obtained a Monte Carlo estimate $\mu \approx 117$ for the growth constant in the above case, and conjectured that the number of such objects grows as $const. \mu^n \cdot n^{-3/2}.$ Note that these structures are three-dimensional. Here, we have analysed the sequence and estimated the growth constant to be $\mu = 117.25 \pm 0.05,$ in excellent agreement with the Monte Carlo estimate, and also provide supporting evidence that the exponent is $-3/2,$ as conjectured.

In 2016 one of us (RN) \cite{N16} studied {\em flat} LEGO structures, which are essentially two-dimensional objects. He considered the number of {\em interlocking} structures one could make with $w \times 1$ LEGO tiles, constrained so that all the tiles are parallel. 
%, with $w=2,\, 3,\, 4,\, 5$ 
These were counted up to size 25 for $w=2,$ and the counts are given in the OEIS \cite{OEIS} as sequence A319156. The growth constant was estimated as $\mu=5.203,$ and it was suggested that these structures grow as $const. \mu^n/n.$

We have written a computer program that allows us to find four further terms for the $2 \times 1$ LEGO problem just discussed, and to generate data for $w \times 1$ flat structures for $w \le 10$ (the sequences for $w > 2$ are not yet in the OEIS \cite{OEIS}).
%, as well as generalisations of these structures that allow them to be placed end-to-end, that is, not necessarily interlocking. 

We have analysed the data and are able to provide  precise (conjectural) estimates of the asymptotic growth for all structures of size $w \le 10.$ 

Furthermore, by considering how these structures grow, not by the number of tiles of a given size $w$ but by the size of the tiles for a given number of tiles $n$, we can prove that this gives rise to polynomials of degree $n-1,$ and we give the first 14 such polynomials, which therefore gives the first 14 coefficients for {\em any} size tiles $w \times 1.$

In the next section we briefly discuss our enumeration algorithm, and in subsequent sections we describe our methods of analysis. These are given in detail for the case of $2 \times 1$ structures, and then applied {\em mutatis mutandis} to other values of $w$. In those cases we simply quote the results. All our results are summarised in the conclusion.

\section{Enumeration program}
We have developed an algorithm for enumerating flat LEGO structures based on Jensen's transfer-matrix algorithm for enumerating  polyominoes \cite{EJ09}.
We can view flat LEGO structures as a special subset of polyominoes by considering each flat LEGO tile of length $w$ as being composed of $w$ adjacent cells. 
%As most polyominoes do not correspond to a LEGO structure, a number of obvious changes and optimizations must be made in order to effectively enumerate the LEGO structures.
\newline
The algorithm is described in detail in \cite{N16} so we will just present a short description here:
\newline
Counting all the LEGO structures one at a time takes too long so the algorithm doesn't construct them directly. Instead it utilizes bottom-up dynamic programming, where efficient (re)use of recurring subproblems is facilitated through extensive memory use. Given a number of blocks $n$ of width $w$ the program considers the 2D space of size $1+n\,(w-1)\times n$ in which all the LEGO structures can fit and then scans the space one cell at a time, left to right and bottom to top, while maintaining just enough information about every possible LEGO structure that can be constructed within the boundary of the scanned cells. The structures are not stored in memory directly but are instead continuously collapsed into equivalence classes defined by the topology of the structures at that boundary. For each of these boundary configurations we store the number of ways it can be reached for different numbers of blocks. In that way we 'build' all the structures at the same time, dealing only with boundary configurations instead of individual structures.
\newline
We also get a significant speed boost by continuously pruning away any boundary configuration that can be seen to not end up contributing to the number of connected structures. This is pretty straightforward since we already need to keep track of how the different parts of the boundary configurations are connected.
\newline
\newline
Utilising this program, we obtained data for $w \times 1$ structures for $w \in [2,10],$
and this data is given in Appendix A.

\section {$2 \times 1$ tiles}
Here we consider the number of ways to build a contiguous structure with $n$ LEGO blocks of size $2 \times 1$ which is flat, i.e., with all blocks in a parallel position. Recall that LEGO structures comprise interlocking blocks, so two blocks placed end-to-end don't count, as they are not interlocked. For clarity, we show in Fig. \ref{fig:1} all 11 structures of three such tiles.
\begin{figure}
    \centering
    \includegraphics[width=1\linewidth]{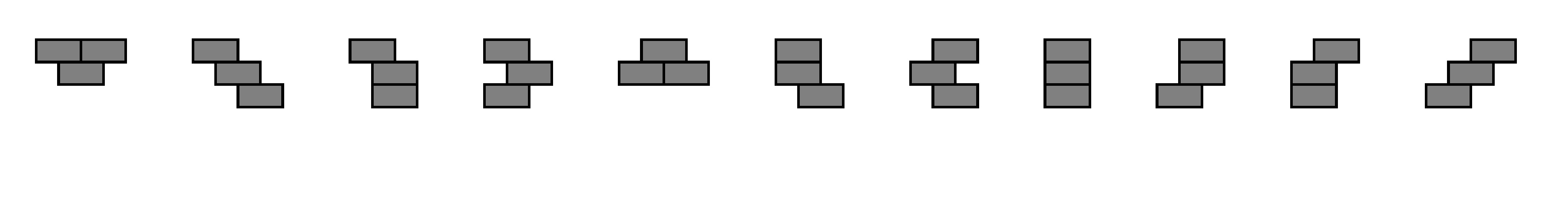}
    \caption{11 orientations of three $2 \times 1$ tiles}
    \label{fig:1}
\end{figure}

The first 25 coefficients are given in the OEIS as the series A319156, first obtained by Nilsson \cite{N16} in his Master's thesis. In Table \ref{tab:1a} we give 4 further coefficients. 

\subsection{Predicting coefficients}
In all cases we have used {\em series extension} to predict further terms (approximately) from the known terms, using differential approximants,  so as to obtain a longer series for analysis. The full theory of series extension is given in \cite{G16}. 

As a demonstration of the effectiveness of this method, assume that we actually only know the first 20 terms in the $2 \times 1$ case, and let us use these to predict the next 9 terms. The results are shown in Table \ref{tab:2}, where it can be seen that the first predicted coefficient is given with an accuracy of more than 8 significant digits. The error increases more or less steadily with order, but even so the worst case error is in the seventh significant digit. This level of precision is adequate for simple ratio plots, as used to estimate the growth constant, and we will use the extended data in each case.

The predicted coefficients are an average over many differential approximants. Based on  many numerical experiments, we take 3 standard deviations as an estimate of the error in the predicted coefficients. This then allows us to estimate how many predicted terms to use in the case when the predicted terms are unknown. We require the standard deviation to be no greater than $3 \times 10^{-6}$ the size of the predicted coefficient.
\begin{table}
    \centering
    \begin{tabular}{|c|c|c|}
    \hline
    Actual coefficient & Predicted coefficient& Fractional error\\
   \hline
 13378627003520  & 13378627126073&9.7E-9 \\
 66495716465315  &66495718454582 & 3.0E-8\\
 331167284581601  &331167311912757 &8.2E-8 \\
 1652340114446553  &1652340414988523 &1.8E-7 \\
 8258197397705302  & 8258200472923852&3.7E-7 \\
 41337852343827210  & 41337882727160076&7.4E-7 \\
 207222462319935608  & 207222762341920980&1.4E-6 \\
 1040176220193951923  &1040179261861774903 & 2.9E-6\\
 5227785863956950802 &5227792460836921059 & 1.3E-6\\
 \hline
    \end{tabular}
    \caption{Coefficient prediction of last 9 terms, $2 \times 1$ tiles.}
    \label{tab:2}
\end{table}

We show in some detail how we analyse this series, but for other values of $w$ we will just quote the results, obtained by applying the same methods of analysis. 

From sub-additivity and Fekete's lemma \cite{F23} it follows that the number, $a_n$ of such structures of size $n$ grows exponentially, and we denote the growth constant by $\mu.$ In the absence of any evidence to the contrary, we tacitly assume the usual power-law dependence, that is $a_n \sim A\mu^n\cdot n^g,$ where $g$ is the {\em critical exponent,} and $A$ is the {\em critical amplitude.} While it is provable that $\lim_{n \to \infty} a_n^{1/n} \to \mu,$ we make the stronger assumption, unproved but undoubtedly true, that the ratios $r_n = a_n/a_{n-1}$ also have $\mu$ as their limiting value.

Indeed, it follows from the assumed power-law form that $$r_n = \mu \left (1 + \frac{g}{n} + O(1/n^2) \right ),$$ and it is this behaviour that lies at the heart of the ratio method of series analysis. Strictly speaking one should write the correction term as $o(1/n),$ but if one has a pure power-law, it is the case that the correction terms are given by increasing powers of $1/n.$ From this expression for the ratios, it follows that plotting the ratios against $1/n$ should, for sufficiently large values of $n,$ result in a straight line with gradient $g\mu,$ and intercept at $1/n=0$ of $\mu.$ 

We show this plot in Fig. \ref{fig:r1}, and it can be seen to be going to an intercept around 5.20. We can improve on this by calculating the so-called {\em linear intercepts,} which effectively eliminates the term of order $g/n$ in the above expression for the ratios. 
We define the linear intercepts as $$l_n=n\cdot r_n - (n-1) \cdot r_{n-1} = \mu(1+O(1/n^2)).$$
In Fig \ref{fig:r2} we show the linear intercepts plotted against $1/n^2,$ which can be seen to be approaching a limit very close to 5.203.

\begin{figure}[ht!] 
\begin{minipage}[t]{0.45\textwidth} 
\centerline {\includegraphics[width=\textwidth]{ 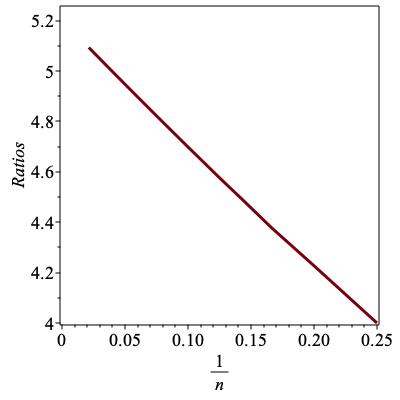}}
\caption{Ratios plotted against $1/n.$} 
\label{fig:r1}
\end{minipage}
\hspace{0.05\textwidth}
\begin{minipage}[t]{0.45\textwidth} 
\centerline{\includegraphics[width=\textwidth]{ 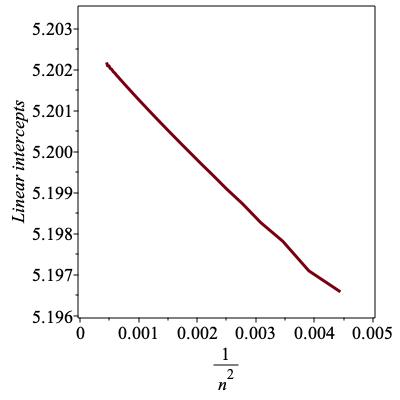}}
\caption{Linear intercepts plotted against $1/n^2.$} 
\label{fig:r2}
\end{minipage}
\end{figure}

We can refine this further if we know the value of the critical exponent $g,$ and this can be estimated from the expression given above for the ratios. Reverting that expression, one obtains:
$$g_n=n\left ( \frac{r_n}{\mu} - 1 \right) = g(1+O(1/n)).$$

\begin{figure}[ht!] 
\begin{minipage}[t]{0.45\textwidth} 
\centerline {\includegraphics[width=\textwidth]{ 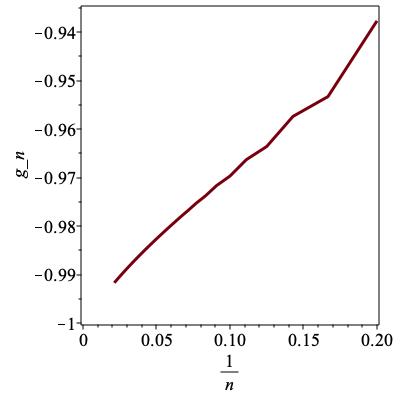}}
\caption{Exponent estimates $g_n$ plotted against $1/n.$} 
\label{fig:g1}
\end{minipage}
\hspace{0.05\textwidth}
\begin{minipage}[t]{0.45\textwidth} 
\centerline{\includegraphics[width=\textwidth]{ 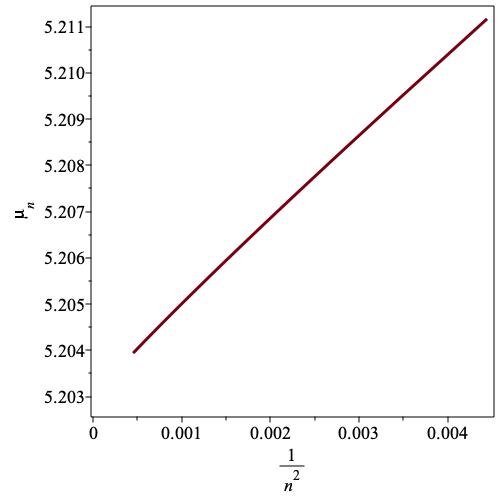}}
\caption{Refined estimate of $\mu$ plotted against $1/n^2.$} 
\label{fig:r3}
\end{minipage}
\end{figure}

In Fig. \ref{fig:g1} we show the exponent estimates, constructed assuming $\mu=5.203,$ and it is clear that the plot is
going to a limit $g \approx -1.$ If, as we conjecture, $g=-1$ then this corresponds to a logarithmic singularity of the generating function, which is not surprising as that is what is conjectured for two-dimensional polyominoes \cite{G09}, and these structures are a subset of polyominoes. If we assume that $g=-1,$ it follows from the expression for the ratios that
$$ \frac{n\cdot r_n}{n-1} =\mu + O(1/n^2).$$ This refined estimate of $\mu$ plotted against $1/n^2$ is shown in Fig. \ref{fig:r3}, and is also consistent with our earlier estimate of $$ \mu \approx 5.203.$$

We can eliminate the term $O(1/n^2)$ from this refined estimate, just as we eliminated the term $O(1/n)$ from the original ratios. We do this by calculating
$$\frac{n^3 r_n}{(n-1)(2n-1)} - \frac{(n-1)^3 r_{n-1}}{(n-2)(2n-1)} =\mu+O(1/n^3).$$ This quantity is shown plotted against $1/n^3$ in Fig \ref{fig:r4}, and we can now refine our estimate of the growth constant to $\mu \approx 5.2030.$

To go further, we assume that $$r_n \approx \mu\left (1 -\frac{1}{n} +\frac{c_1}{n^2} +\frac{c_2}{n^3}\right ).$$ There are three unknowns on the r.h.s., $\mu,$ $c_1$ and $c_2.$ We fit successive triples of coefficients, $a_{k-1},\, a_k,\, a_{k+1}$ to this assumed form, increasing $k$ until we run out of terms, thereby yielding estimates of the three unknowns. In Fig. \ref{fig:r5} we show the resulting estimate of $\mu,$ and this allows us to make the more refined estimate, $\mu = 5.20300 \pm 0.00005.$ 

\begin{figure}[ht!] 
\begin{minipage}[t]{0.45\textwidth} 
\centerline {\includegraphics[width=\textwidth]{ 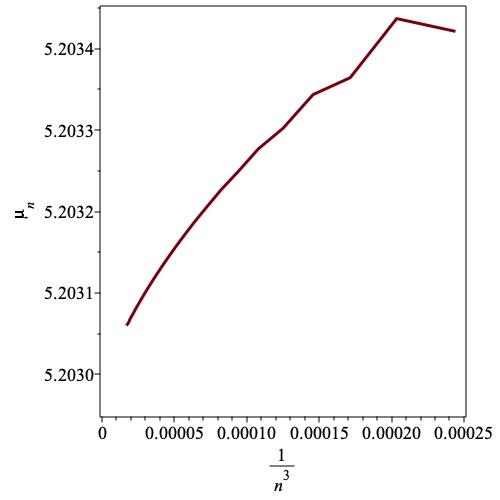}}
\caption{Refined estimate of $\mu$ plotted against $1/n^3.$} 
\label{fig:r4}
\end{minipage}
\hspace{0.05\textwidth}
\begin{minipage}[t]{0.45\textwidth} 
\centerline{\includegraphics[width=\textwidth]{ 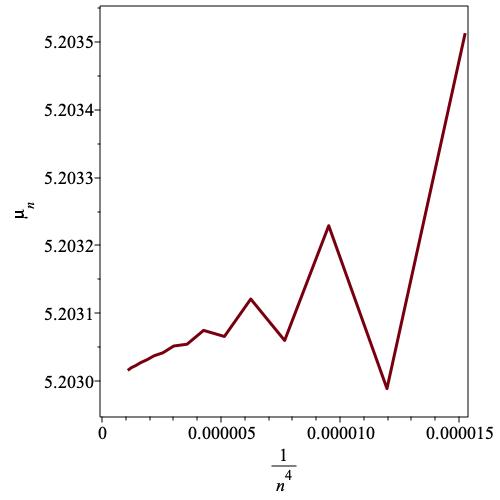}}
\caption{Refined estimate of $\mu$ plotted against $1/n^4.$} 
\label{fig:r5}
\end{minipage}
\end{figure}
We can further improve on this by fitting four successive terms to $$r_n \approx \mu\left (1 -\frac{1}{n} +\frac{c_1}{n^2} +\frac{c_2}{n^3}+\frac{c_3}{n^4}\right ).$$
The result is shown in Fig. \ref{fig:r6}, and this allows us to make our final estimate, $\mu = 5.203000 \pm 0.000005.$ 
Finally, in Fig \ref{fig:cmu} we show estimators of $c_1,$  which seems to be around  $c_1 \approx 0.41.$ These estimators follow from the expression for the ratios, so that
$$c_1 \sim \left ( \frac{r_n}{\mu}-1+\frac{1}{n} \right )n^2.$$
\begin{figure}[ht!] 
\begin{minipage}[t]{0.45\textwidth} 
\centerline {\includegraphics[width=\textwidth]{ 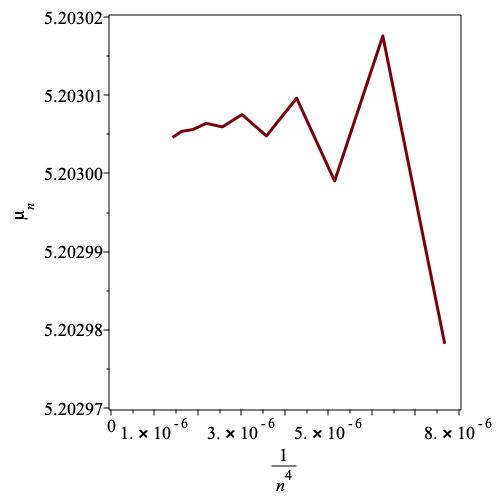}}
\caption{Refined estimate of $\mu$ plotted against $1/n^4.$} 
\label{fig:r6}
\end{minipage}
\hspace{0.05\textwidth}
\begin{minipage}[t]{0.45\textwidth} 
\centerline{\includegraphics[width=\textwidth]{ 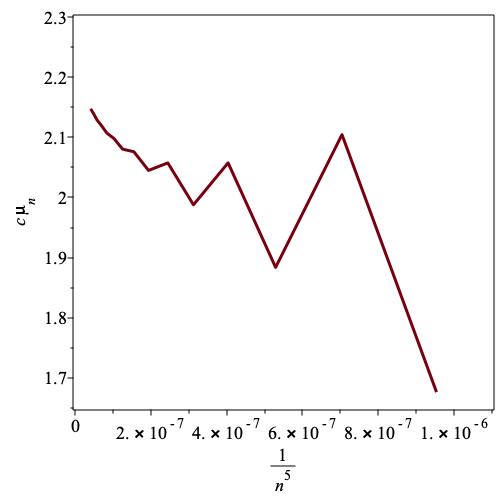}}
\caption{Estimators of $c_1\cdot \mu$ plotted against $1/n^5.$} 
\label{fig:cmu}
\end{minipage}
\end{figure}

Thus we find that the ratios behave as

$$r_n = \mu\left (1 -\frac{1}{n} + \frac{c_1}{n^2} + O(n^{-3})\right ),$$
where $c_1\approx 0.41.$ 
It follows that the generating function
$$A(x)=\sum_n a_n x^n  \sim \frac{C}{x}\log(1-\mu \cdot x),$$ where $C = -0.26022 \pm 0.00005.$ We estimated $C$ by extrapolating the sequence $\frac{a_n}{\mu(1-1/n)},$ and making the necessary transformation from coefficient amplitude to generating function amplitude.

Incidentally, if one tests the available data for log-convexity, one finds that the sequence $a_{n-1}a_{n+1}-a_n^2$ is both positive and (rapidly) monotonically increasing. Assuming that this property persists, it follows that $a_n/a_{n-1}$ gives an increasing sequence of lower bounds on the growth constant. The available coefficients then give the (non-rigorous) lower bound $\mu > 5.0196.$ This property (log-convexity) appears to be true for all the LEGO sequences considered in this article.

\section {$w \times 1$ tiles}
As mentioned above, we have also enumerated the number of LEGO structures comprised of $w \times 1$ tiles, again with all tiles oriented parallel, for $w = 3,\cdots,10.$
The counts are shown in Table \ref{tab:1a}, and we have analysed the corresponding series in exactly the same way as described above for structures comprising $2 \times 1$ tiles. 

The generating function behaves similarly in all cases, with a logarithmic singularity, and a growth constant that grows roughly linearly with $w$. We will say more about this subsequently.
More precisely, for $w=3$, we estimate
$\mu = 8.84035 \pm 0.00005,$ and $g=-1.$ The ratios behave as,
$$r_n = \mu\left (1 -\frac{1}{n} + \frac{c}{n^2} + O(n^{-3})\right ),$$
where $c\approx 0.43,$ which is slightly higher than the value found for the $2 \times 1$ case. For the generating function,
we estimate
$$A(x) =\sum_n a_n x^n \sim C\log(1-\mu \cdot x)/x,$$ where $C = -0.1485 \pm 0.0005.$

Again, assuming log-convexity, which is well-supported by the data, one obtains the (non-rigorous) lower bound $\mu > 8.426.$

A similar analysis for $w\ge 4$ gave the results shown Table \ref{tab:2a}, with the ratio coefficient $c$ for the $w=4$ case estimated as $c \approx 0.47,$ so it appears that this sub-dominant term increases with $w.$

The amplitudes quoted in Table \ref{tab:2a} are coefficient amplitudes. That is to say, the coefficients grow as $\mathcal{A}\mu^n/n,$ while the generating functions behave as $A(x) =\sum_n a_n x^n \sim C\log(1-\mu \cdot x)/x,$ where $C=-\mathcal{A}.$

\begin{table}
    \centering
    \begin{tabular}{|c|c|c|}
    \hline
    $w$ & Growth constant& Amplitude\\
   \hline
 2  & 5.2029985& 0.26022\\
 3  &8.84045 & 0.1485\\
 4  &12.43953 & 0.1050\\
 5  &16.027& 0.08113\\
 6  & 19.609& 0.0662\\
 7  &23.188 & 0.0559\\
 8  &26.767 & 0.04825  \\
 9  &30.341 & 0.0429\\
 10 & 33.917& 0.0382\\
 \hline
    \end{tabular}
    \caption{Results for asymptotics of $w \times 1$ tiles, $w = 2,\cdots, 10.$ The coefficients grow as $c_m \sim {\mathcal A}\cdot \mu^m/m,$ where $\mu$ is the growth constant and ${\mathcal A}$ is the amplitude.}
    \label{tab:2a}
\end{table}

\section{Polynomials}
So far we have analysed the sequences that arise when fixing the tile size and varying the number of tiles. In this section we look at the sequences that arise when we fix the number of tiles $n$ and vary the tile size $w$ instead, an idea first suggested to one of us by Johan Nilsson. In that case the sequences are exact polynomials, as we observe experimentally for lower values of $n$ and which holds generally, as we prove below. In \cite{N16} these polynomials are listed for $n=1..12, $ and here we extend this to $n=14.$
\newline
\newline
\noindent
Let's first take a look at a well-understood special case: the flat LEGO \textit{pyramids} \cite{DE10}. A LEGO pyramid is a LEGO structure in which every tile, except for a single base tile, is resting on top of another tile.
In \cite{DE10} Durhuus and Eilers proved an exact expression for the number $\rho_n(w)$ of flat LEGO pyramids that can be built with $n$ tiles of length $w$:

\begin{theorem}
$\rho_n(w) = \binom{wn-1}{n-1}.$
\label{thm:pyra}
\end{theorem}
We don't expect to find an equally elegant solution for the more general structures. However, the experimental data suggests that the general case shares many nice properties with the pyramids which we shall discuss in this section.

Let's rewrite the formula:
\begin{equation} \label{pyrapoly}
\rho_n(w)=\binom{wn-1}{n-1}=\frac{\prod_{i=1}^{n-1}(wn-i)}{(n-1)!}=\frac{n^{n-1}}{(n-1)!}\prod_{i=1}^{n-1} \left( w-\frac{i}{n} \right).
\end{equation}
We see that for fixed $n\in\mathbb{N}$ this is an $(n-1)$th degree polynomial in $w,$ with leading polynomial coefficient $\frac{n^{n-1}}{(n-1)!}$ and roots $\frac{1}{n},\frac{2}{n},...,\frac{n-1}{n}$.
\newline
The first few such polynomials are:
$$\rho_2(w)=2w-1,$$
$$\rho_3(w)=\frac{9}{2}w^2-\frac{9}{2}w+1,$$
$$\rho_4(w)=\frac{32}{3}w^3-16w^2+\frac{22}{3}w-1.$$
It will turn out to be useful to rewrite these polynomials in the {\em combinatorial representation}:
\[
\rho_n(w) = \sum_{k=1}^{n} a_{n,k} \binom{w-1}{k-1}.
\]
The first few such polynomials are:
$$\rho_2(w)=2 \binom{w-1}{1}+ \binom{w-1}{0}.$$
$$\rho_3(w)=9 \binom{w-1}{2}+9 \binom{w-1}{1}+ \binom{w-1}{0}.$$
$$\rho_4(w)=64 \binom{w-1}{3}+96 \binom{w-1}{2}+34 \binom{w-1}{1}+ \binom{w-1}{0}.$$
\newline
\noindent
Traditionally, in this representation, each binomial coefficient is associated with a distinct class of whatever objects are being counted. In many cases these classes are given by a simple characteristic such as the {\em width} or {\em height} of the structures being counted.

We observe that a similar representation holds for the LEGO structures we are studying. That is to say, we can write
\[
p_n(w) = \sum_{k=1}^{n} a_{n,k} \binom{w-1}{k-1},
\]
where the coefficients $a_{n,k}$ will be greater than or equal to the corresponding coefficients in the case of pyramids.
We have verified this polynomial structure experimentally for $n \le 10$, and we prove below that it holds in general, just as it does for pyramids. By invoking theorem \ref{conj} (given below) we found the polynomials up to $n \le 14.$ 

For the LEGO structures we are counting, the concept of a class is rather more complicated than just width or height. We first define what we call a {\em type,}
a definition which applies equally to pyramids as to the full LEGO structures we are studying
-- the only thing that changes when going from pyramids to the full LEGO structures is the coefficient $a_{n,k}$ multiplying each binomial coefficient.

%Here is the definition of a {\em type.} 
We first number the blocks from bottom-left to top-right $0,\,1,\,\cdots,\, n-1.$ We take the left-most co-ordinate of the leftmost, bottommost block to be 0. Let $b_i$ be the horizontal offset of each block from 0, so $b_0=0$ always. This labelling uniquely defines the structure, assuming it is a constructible LEGO structure of the type we are considering.

\begin{figure}[h]
\centering
%
% --- Figure (a) ---
\begin{tikzpicture}[scale=0.6]
  \fill[gray!40] (5,2) rectangle (8,3);
  \fill[gray!40] (0,1) rectangle (6,2);
  \fill[gray!40] (1,0) rectangle (4,1);
  \draw[very thick] (5,2) rectangle (8,3);
  \draw[very thick] (0,1) rectangle (6,2);
  \draw[very thick] (3,1) -- (3,2);
  \draw[very thick] (1,0) rectangle (4,1);
  \node at (4, -0.6) {(a)};
\end{tikzpicture}%
\hspace{2cm}%
% --- Figure (b) ---
\begin{tikzpicture}[scale=0.7]
  \fill[gray!40] (7,2)  rectangle (11,3);
  \fill[gray!40] (0,1)  rectangle (8,2);
  \fill[gray!40] (2,0)  rectangle (6,1);
  \draw[very thick] (7,2)  rectangle (11,3);
  \draw[very thick] (0,1)  rectangle (8,2);
  \draw[very thick] (4,1) -- (4,2);
  \draw[very thick] (2,0)  rectangle (6,1);
  \node at (5, -0.6) {(b)};
\end{tikzpicture}%
\caption{Two LEGO structures (also pyramids) with $n=4,$ and (a) $w=3,$ and (b) $w=4.$}
    \label{fig:ex}
\end{figure}
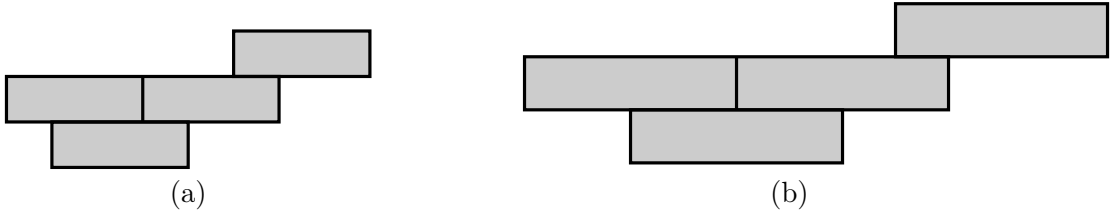

In fig. \ref{fig:ex} we show two LEGO structures (which are also pyramids) with $n=4,$ and, in (a) $w=3$  with  $b_0=0,$ $b_1=-1,$ $b_2=2,$ $b_3=4,$ and in (b) with  $b_0=0,$ $b_1=-2,$ $b_2=2,$ $b_3=5.$

We define the {\em type} as follows:

\begin{definition}[Type]
Given a valid structure with tile size $w$ and offset sequence 
$b_0 = 0, b_1, \ldots, b_{n-1}$, write each offset uniquely as
\[
    b_i = w \cdot q_i + r_i, \quad q_i \in \mathbb{Z},
    \quad r_i \in \{0, 1, \ldots, w-1\},
\]
so that $r_i = b_i \bmod w$ is the residue class of tile $i$ and 
$q_i = \lfloor\frac{b_i}{w}\rfloor$. Relabel the distinct residue classes 
in increasing order as $0, 1, \ldots, k-1$, where $k$ is the number 
of distinct residue classes, to obtain the normalised residue sequence 
$(r_0^{\mathrm{norm}}, r_1^{\mathrm{norm}}, \ldots, r_{n-1}^{\mathrm{norm}})$.

The \emph{type} of the structure is the pair
\[
\bigl(r_0^{\mathrm{norm}}, r_1^{\mathrm{norm}}, \ldots, r_{n-1}^{\mathrm{norm}}\bigr)
\quad\text{and}\quad
\bigl(q_0, q_1, \ldots, q_{n-1}\bigr).
\]
Two structures, possibly with different tile sizes, are of the 
\emph{same type} if and only if both sequences agree.
\end{definition}
The number of types with $n$ tiles and with exactly 
$k$ distinct elements in their normalised residue sequence is denoted 
$a_{n,k},$ and is independent of $w.$

Note that if a residue sequence is, say $(0,3,3,2)$ or $(0,7,7,4),$ then the normalised residue sequence would be $(0,2,2,1)$ in both cases.

For example, consider fig. \ref{fig:ex}(a). We have ${\bf b}=(0,-1,2,4).$ So the residue sequence (already normalised) is $(0,2,2,1),$  and the $q$ sequence is ${\bf q}=(0,-1,0,1).$  

In fig. \ref{fig:ex}(b) we have ${\bf b}=(0,-2,2,5).$ so the (normalised) residue sequence is again $(0,2,2,1),$ and the $q$ sequence is again ${\bf q}=(0,-1,0,1).$ So these two figures are of the same type. In both cases there are three distinct residue classes, so $k=3.$

We can now state the main theorem:
\begin{theorem}
For each fixed $n \in \mathbb{N}$, the function $p_n(w)$ is a polynomial
in $w$ of degree $n - 1$.
\end{theorem}

\begin{proof}
We partition all valid $n$-tile structures according to their type.
%Since the offset sequence $(b_0, \ldots, b_{n-1})$ takes values in a
%finite combinatorial set for each $n$ (up to the equivalence above),
%there are finitely many types, and their count is independent of $w$.

We claim that each type of offset complexity $k$ gives rise to exactly
$\binom{w-1}{k-1}$ distinct valid structures for each $w \geq k$.
Indeed, one residue class is fixed at $0$ (since $b_0 = 0$), and to
realise a given type one must choose which $k - 1$ further residue
classes from $\{1, 2, \ldots, w-1\}$ are occupied. Any such choice
yields a valid structure whose offset sequence is in the given
equivalence class, and distinct choices yield inequivalent sequences.
The number of such choices is $\binom{w-1}{k-1}$.

Letting $a_{n,k}$ denote the number of types of offset complexity $k$,
and summing over all types, we obtain
\[
    p_n(w) = \sum_{k=1}^{n} a_{n,k} \binom{w-1}{k-1}.
\]
The coefficients $a_{n,k}$ are non-negative integers independent of $w$.
Since $\binom{w-1}{k-1}$ is a polynomial in $w$ of degree $k-1$, the
right-hand side is a polynomial in $w$ of degree at most $n - 1$.

It remains to show that the degree is exactly $n - 1$, i.e.\ that
$a_{n,n} > 0$. Consider the staircase configuration defined by $b_i = i$
for $i = 0, 1, \ldots, n-1$. For $w \geq n$ these offsets are all
distinct modulo $w$, so this structure has offset complexity $n$, and
hence $a_{n,n} \geq 1$. The term $k = n$ contributes $a_{n,n}
\binom{w-1}{n-1}$, a polynomial of degree $n-1$ with positive leading
coefficient, so $\deg p_n(w) = n - 1$.
\end{proof}

Some examples may help to make this definition clearer. For example, in Fig. \ref{fig:32} we show all 11 possible LEGO structures with $w=2$ and $n=3.$
\begin{figure}[htbp]
    \centering
    \includegraphics[angle=270,width=1.0\linewidth]{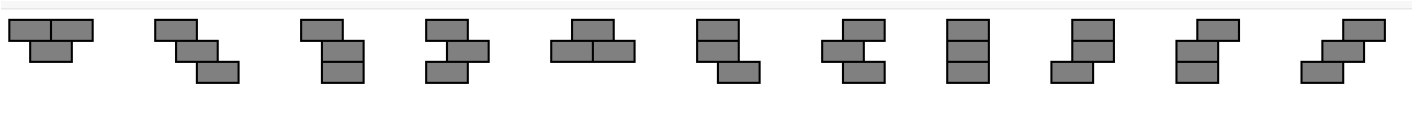}
    \caption{All 11 Lego structures with $w=2,$ and $n=3.$}
    \label{fig:32}
\end{figure}
We can partition all valid structures by their offset complexity $k$.  The eighth entry has offset signature $(0,0,0),$ and hence offset complexity $k=1.$ All the others have offset complexity $k=2,$ as it is not possible to have offset complexity $k > w-1=2,$ as noted above. So we see that $a_{3,1}=1$  and $a_{3,2}=10.$ 

When $w=3$ and $n=3$ there are 31 LEGO structures, as shown in Fig. \ref{fig:33}, and we have the result
$$p_3(w)=10 \binom{w-1}{2}+10 \binom{w-1}{1}+ \binom{w-1}{0}.$$
There are twenty structures with
offset complexity $k=2,$ and ten with offset complexity $k=3.$ The count 20 is just the product $10 \binom{w-1}{1},$ and the count 10 is just the product $10 \binom{w-1}{2},$ (recall that $w=3$ in this case). That is, $a_{3,2}=10$
as each structure with offset complexity $k=2$ now occurs with multiplicity 2, as there are exactly two structures of each type. 

As an example, the first structure in Fig. \ref{fig:32} is of the same {\em type} as the last two structures in Fig. \ref{fig:33}. And when $w=4,$ that first structure in Fig. \ref{fig:32} will spawn 3 structures of the same type. (In all cases, the normalised residue sequence is $(0,1,1)$ and the $q$ sequence is $(0,-1,0).$)

The point is that $a_{n,k}$ is the number of distinct structure types comprising $n$ tiles, with offset complexity $k.$ Each type occurs with multiplicity $\binom{w-1}{k-1}.$
To determine the value of $a_{n,k},$ note that $\binom{w-1}{k-1}=1$ when $w=k,$ so one just needs to count the number of structures with offset complexity $k=w$ when dealing with tiles of size $w \times 1.$
Note too that $a_{n,k}=0$ when $w < k,$ (as in that case there are no possible offset sequences such that  $k=w$), but this is automatically accounted for as the binomial coefficient $\binom{w-1}{k-1}$ vanishes in that case.

Summarising, we have shown that:
\[
\rho_n(w) = \sum_{k=1}^{n} a_{n,k} \binom{w-1}{k-1},
\]
for pyramids, and for general LEGO structures, we have:
\[
p_n(w) = \sum_{k=1}^{n} a_{n,k} \binom{w-1}{k-1},
\]
where of course the coefficients $a_{n,k}$ will in general be different for the two situations, as pyramids are a subset of the LEGO structures we are considering.

\vspace{-3.5cm}
\begin{figure}[htbp]
\centering
  \includegraphics[angle=270,width=1.0\linewidth]{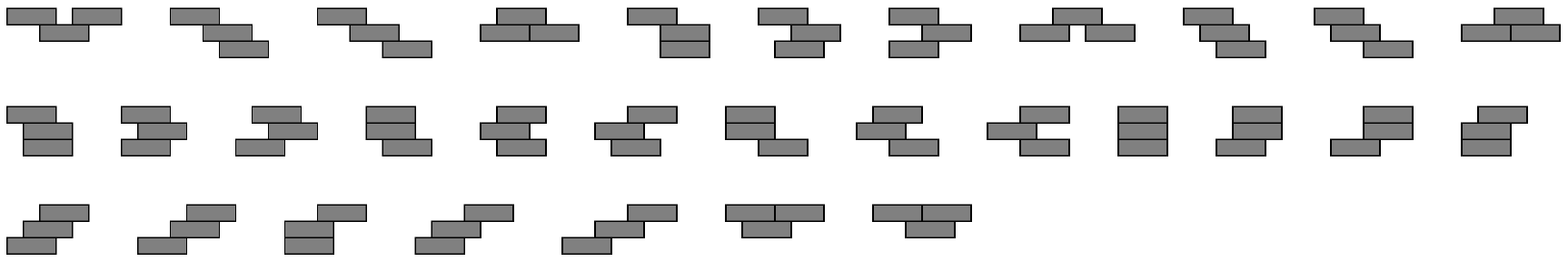} 
  \vspace{-4cm}
\caption{Thirty-one allowable structures of three $3 \times 1$ blocks.}
\label{fig:33}
\end{figure}

Using the data from the enumeration program we find
$$p_3(w)=10 \binom{w-1}{2}+10 \binom{w-1}{1}+ \binom{w-1}{0}.$$ 

$$p_4(w)=82 \binom{w-1}{3}+123 \binom{w-1}{2}+43 \binom{w-1}{1}+ \binom{w-1}{0},$$ 

$$p_5(w)=950\binom{w - 1}{4} +1900 \binom{w-1}{3}+1135 \binom{w-1}{2}+185 \binom{w-1}{1}+ \binom{w-1}{0},$$ 
$$p_6(w)=14260\binom{w - 1}{5}+35650\binom{w - 1}{4} +30142 \binom{w-1}{3}+9563 \binom{w-1}{2}+813 \binom{w-1}{1}+ \binom{w-1}{0},$$ 
\begin{multline*}
p_7(w)=263728\binom{w - 1}{6}+ 791184\binom{w - 1}{5}+865575\binom{w - 1}{4} +\\+412510 \binom{w-1}{3}+78046 \binom{w-1}{2}
+3655 \binom{w-1}{1}+ \binom{w-1}{0},$$ 
\end{multline*}
\begin{multline*}
p_8(w)=5799524\binom{w - 1}{7} +20298334\binom{w - 1}{6}+ 27257646\binom{w - 1}{5}+17398280\binom{w - 1}{4} +\\+5252794 \binom{w-1}{3}+630078 \binom{w-1}{2}
+16730 \binom{w-1}{1}+ \binom{w-1}{0}.$$ 
\end{multline*}
From our enumerations we can directly determine the first ten polynomials $p_1..p_{10}$. These have been verified with a varying number of extra data points. For $n<7$ we have enumeration data for block lengths up to at least $w=25$. For larger $n$ we can't go quite as far.
%5799524*binomial(w - 1, 7) + 20298334*binomial(w - 1, 6) + 27257646*binomial(w - 1, 5) + 17398280*binomial(w - 1, 4) + 5252794*binomial(w - 1, 3) + 630078*binomial(w - 1, 2) + 16730*binomial(w - 1, 1) + binomial(w - 1, 0)
%[TODO: maybe add a pretty plot of the polynomials to motivate the idea of symmetry]
\newline \noindent
Several characteristics of these polynomials can be observed. Recall the structure, $$p_n(w)=\sum_{k=1}^{n} a_{n,k} \binom{w-1}{k}.$$
We observe that $$a_{n,n-1} = \frac{n-1}{2}a_{n,n},\,\,\, a_{n,0}=1,$$ and that $$\sum_{k=1}^{n} (-1)^{k} a_{n,k} = (-1)^{n+1}.$$
The last observation is a consequence of a very useful symmetry property that applies. 

\begin{theorem}\label{conj}

$$\forall n\in\mathbb{N}\hspace{15pt} p_n(1-w) = (-1)^{(n-1)}p_n(w).$$
\label{thm:evenodd}
\end{theorem}
\begin{proof}

We will break our collection of LEGO configurations into a number of equivalence classes, each corresponding to a \textit{connected red-black graph} $G$ as defined below. A class will consist of: 

\begin{itemize}
\item A tuple $(m_1, ... , m_k)$ with $m_1+ \cdots+m_k=n$ where $m_i$ is the number of bricks on the $i$th layer. Layers are numbered from the bottom (layer 1) up. 
\item For each of the $m_1m_2+m_2m_3 + ... +m_{k-1}m_k$ pairs of bricks on adjacent layers the disjoint possibilities (1) they are joined, (2) the top brick is strictly to the left of the bottom brick, (3) the top brick is strictly to the right of the bottom brick.
\end{itemize}

Note that here a \textit{class} is a coarser notion than the notion of \textit{type}, defined above.

From such data one can construct a directed graph with red and black edges -- the vertices correspond to bricks, while edges encode the relative-position constraints between bricks, i.e. horizontal-offset inequalities. One constructs a red edge from the vertex corresponding to the brick furthest to the right to the vertex of the brick to the left in cases (1) and (3), and constructs a black edge going both ways in case (2). For any two adjacent bricks on the same layer, one is constrained to lie strictly to the right of the other, and we similarly add a red edge to enforce this constraint.

Note also that such a directed graph restricted to its black edges is strongly connected, as our LEGO configuration itself is connected. 

Expressed more formally as follows, we can count structures in a class by points in a polytope as shown below, bringing the problem into the realm of Ehrhart reciprocity \cite{BR, S74}.%  We now define a polynomial associated with such a graph. 

%We will then useConsequently, counting structures translates to counting integer points in a polytope, and the reflection symmetry is then
\begin{definition}
Consider $H=(V=\{v_0, ... ,v_{n-1}\},E_B,E_R)$ a directed graph with black edges (denoted by $E_B$) and red edges (denoted by $E_R$). 
 Assume also that $G$ restricted to its black edges is strongly connected. We call such a graph a \textit{connected red - black graph}. For $w>0$ an integer, we let $g_H(w)$ be the number of tuples $(a_1, ... ,a_{n-1}) \in \mathbb{Z}^{n-1}$ satisfying 

\begin{itemize}
\item $a_0:=0$
\item $a_{i} <  a_{j} +w : (i,j) \in E_B$
\item $a_{i} \ge a_{j} +w : (i,j) \in E_R$
\end{itemize}

Note that the strong connectedness (i.e, the existence of a directed path in both directions from any vertex to the root) is a sufficient  condition for this to be finite: if there is a path of length $k$ from $v_i$ to $v_j$ we have $a_i < a_j+kw$ so in particular for all $i,$ $|a_i|<(n-1)w$.

%Note the change in notation: I'll use H for red-black graphs and G for particular black graphs but its the same function

\end{definition}

%Each $(n-1)$-tuple is a point in the set of valid integer labellings. For a fixed graph $G$ and fixed $w,$ there are generally many valid tuples, and these are counted by $g_G(w).$ That is to say,
%\[g_G(w) = \#\{a : a_i < a_j + w \,\,\text{for}\,\, (v_i,v_j) \in E\} = \#\{a : a_i \le a_j + w - 1\},\] as the ${a_i}'s$ are integers.
%Loosely speaking, the tuple's job is to turn the number of LEGO structures into the number of lattice points in a polytope, which then unlocks the Ehrhart machinery. 

If a class corresponds to some graph $H,$ then $g_H(w)$ counts the number of LEGO configurations in that class. Furthermore, $p_n(w)$ is a sum of functions of the form $g_H(w)$ where $H$ varies over some connected red - black graphs on $n$ vertices.

It now suffices to prove that $g_H$ can be extended to a polynomial satisfying $g_H(1-w)=(-1)^{n-1}g_H(w)$.

In full generality this is not yet quite enough for us to use basic Ehrhart duality. However, we can use it in a special case:

% Note: I'm not sure but wouldn't be surprised to see more general results of this form in the Ehrhart literature

\begin{definition}

For any strongly connected rooted directed graph $G=(V,E)$ we can define $g_G$ by considering it to be a connected red-black graph without any red edges, that is, for $w>0$ an integer, we let $g_H(w)$ be the number of tuples $(a_1, ... ,a_{n-1}) \in \mathbb{Z}^{n-1}$ satisfying 

\begin{itemize}
\item $a_0:=0$
\item $a_{i} <  a_{j} +w : (i,j) \in E_B$
\end{itemize}

We call these \text{connected black graphs}.
\end{definition}

\begin{lemma}
For any connected black graph $G$, $g_G$ may be extended to a polynomial (which we will also call $g_G$ for brevity) satisfying
\[g_G(1-w)= (-1)^{n-1}g_G(w).\]
\end{lemma}
\begin{proof}

Now, the constraints $x_i \le x_j+1$ for $(v_i,v_j) \in E$ carve out an integer polytope $P_G$ in $\mathbb{R}^{n-1}=\{(x_1,...,x_{n-1}): x_i  \in \mathbb{R}\}$ where we define $x_0=0,$ and this is exactly the polytope we dilate. $P_G$ has a nonempty interior (all $x_i = 0$ satisfies $x_i \le x_j + 1$ strictly), hence a neighbourhood of the origin lies in $P_G,$ so it is full-dimensional in $ \mathbb{R}^{n-1}.$
Its Ehrhart polynomials are, for positive integer $w,$
\[L_{P_G}(w)=\#\{a \in {\mathbb Z}^{n-1}:a_i \le a_j+w\}=g_G(w+1),\] since $a_i \le a_j+w$ is $a_i < a_j+w+1,$ i.e. the strict-g at argument $w+1.$ And
\[L_{P^\circ_G}(w)=\#\{a \in {\mathbb Z}^{n-1}:a_i < a_j+w\}=g_G(w).\] Recall that $P_G^\circ$ is the interior of $P_G$.
By Ehrhart  duality \cite{S74, BR} it is the case that 
$L_{P_G^\circ}(w) = (-1)^{n-1}L_{P_G}(-w).$ 
%where $L_P(t)$ is the Ehrhart polynomial for the polytope $P$ for $P=P_G,P_G^\circ$, i.e. for positive integer $t,$ $L_P(t)$ is the cardinality %of $\mathbb{Z}^{n-1}\cap tP$ where $tP=  \{t\vec{v}: \vec{v} \in P\}$. It is certainly the case for positive integer $w$ that
%\[L_{P_G}(w)=g_G(w).\]
This gives \[g_G(w)=(-1)^{n-1}g_G(-w+1)=(-1)^{n-1}g_G(1-w).\] 
By Ehrhart's theorem $L_{P_G}$ is a polynomial, so $g_G(w) = L_{P_G}(w-1)$ agrees with a polynomial for all integers $w \ge 1;$ we take this polynomial as the extension of $g_G.$
Since these agree for all positive integers $w$ and both sides are polynomials, the identity holds identically.
%Since both sides are polynomials in $w,$ this identity holds identically, and substituting $w \to 1-w$ gives $g_G(-w) = (-1)^{n-1} g_G(w-1);$ %replacing $w$ by $w+1$ then yields
%\[g_G(1-w)= (-1)^{n-1}g_G(w)\].
\footnote{A more general setting for this kind of identity seems to be convex integer polytopes whose linear programming duals are themselves integer polytopes -- this is a pleasing symmetry, but it is unclear to the authors what is known about such objects.}
\end{proof}

We would like to extend this result from connected black graphs to connected red - black graphs. Fortunately it will turn out that $g_H$ is a linear combination of polynomials of the form $g_G$ for $G$ a connected black graph.
\begin{lemma}
$g_H(w)$ is a polynomial satisfying 
\[g_H(1-w)=(-1)^{n-1}g_H(w)\]
\end{lemma}
\begin{proof}
By the inclusion - exclusion principle for $G=(V=\{v_0, ... ,v_{n-1}\},E_B,E_R)$ we have
\[g_H(w) = \sum_{T \subseteq E_R} (-1)^{|T|} g_{(V, E_B \cup T)}(w).\]
Because a red edge imposes $a_i \ge a_j + w,$ which is the complement of the black constraint $a_i < a_j + w \,\,(= a_i \le a_j + w-1)$ on that edge - over the integers these two are exactly complementary, with no overlap and no gap. So ``red edge satisfied"= ``black-version of that edge violated". $g_{(V, E_B \cup T)}$ counts tuples satisfying the black $E_B$ constraints, and the black version of the edges in $T,$ and the alternating sum over $T$ sieves out exactly those violating all red edges' black-versions, i.e. satisfying all the red constraints.
So
\[g_H(1-w) = \sum_T (-1)^{|T|} g_{G_T}(1-w) = \sum_T (-1)^{|T|} (-1)^{n-1} g_{G_T}(w) = (-1)^{n-1} g_H(w),\]
and so each of the graphs $G_T=(V,E_B \cup T)$ satisfy

\[g_{G_T}(1-w)=(-1)^{n-1}g_{G_T}(w)\]

That is to say, for each subset $T$ of the red edges, recolour those red edges as black (turning the red $a_i \ge a_j+w$ into the black $a_i \le a_j+w-1$ on that edge), count the resulting all-black graph with $g,$ and add it with sign $(-1)^{|T|}.$ Because each summand $g_{G_T}$ satisfies the $(-1)^{n-1}$ reflection, and the signs $(-1)^{|T|}$ don't depend on $w,$ the whole alternating sum inherits the same reflection symmetry.
\end{proof}

Observing that 
\begin{itemize}
\item $p_n$ is a linear combination of polynomials $g_H$ for $H$ a connected red - black graph on $n$ vertices.
\item For any connected red-black graph $H$ on $n$ vertices $g_H$ is a linear combination of polynomials $g_G$ for $G$ a connected black graph on $n$ vertices.
\item For $G$ a connected black graph on $n$ vertices $g_G$ is a polynomial satisfying $g_G(1-w)=(-1)^{n-1}g_G(w)$
\end{itemize}

%\begin{lemma}
%$p_n$ is a sum of polynomials $g_H.$
%\end{lemma}
%begin{proof}
%We can break the set of Lego polyominoes into classes which are enumerated by polynomials $h_G$: we enumerate the number of bricks in each layer, and label them in some canonical way (e.g. from left to right, one layer at a time). We get a constraint from insisting bricks in the same layer do not overlap and we also specify, for any pair of bricks in adjacent layers, the disjoint possibilities of: (1) they are joined, (2) the top brick is strictly to the left of the bottom brick, (3) the top is strictly to the right. Every structure determines its layer-brick-counts and, for each adjacent pair, exactly one of the three relative positions - so the class is a function of the structure, and the classes are therefore disjoint and exhaustive. Since the classes partition all $n$-tile structures and each class $C$ is counted by some $h_{G(C)}$ (by construction of its constraint graph), summing over classes gives $p_n(w) = \sum_C h_{G(C)}(w)$, a sum of $h_G$ polynomials.
%\end{proof}
%. Note too that the connectedness of our Lego polyominoes means that all graphs constructed here are strongly connected. The point here is that each class of structures is counted by one $g_GG$ hence by an alternating sum of $g_G$'s. Summing over all classes gives $p_n(w)$ as a finite alternating sum of Ehrhart polynomials - every term reflection-symmetric up to the same $(-1)^{n-1}.$ 

We conclude that 
\[ p_n(1-w) = (-1)^{(n-1)}p_n(w).\]

\end{proof}

So $p_{n+1}(w+\frac{1}{2})$ is an even resp. odd polynomial when $n$ is an even resp. odd integer. This follows immediately by shifting $w \mapsto w+\tfrac12$. The same is true for the pyramids and follows from the symmetric placement of the roots $\frac{1}{n},\frac{2}{n},...,\frac{n-1}{n}$ around $\frac{1}{2}$.

As this is true in general we needed only compute half as many points to determine $p_n$. Normally we would need to compute $p_n(1),p_n(2),...,p_n(n)$, but exploiting this symmetry leaves us with the substantially easier task of computing just the first $\lceil\frac{n}{2}\rceil$ numbers $p_n(1),p_n(2),...,p_n(\lceil n/2\rceil)$.
This yields  four more polynomials $p_{11},p_{12},p_{13}$ and $p_{14}$ with the currently available data. The first eight polynomials are given explicitly above, and to
save space, we give the coefficients of the remaining six polynomials $p_9(w), \ldots , p_{14}(w)$ in tabular form in Table \ref{tab:coeffs}.%\footnote{We have been advised by Michal Adamaszek that these results can be proved, arguably more elegantly, using the machinery of Ehrhart theory.}.

\begin{table}[h]
\footnotesize
    \centering
    \begin{tabular}{|c|c|c|c|c|c|c|}
    \hline
    Coefficient &$n=9$&$n=10$&$n=11$&$n=12$&$n=13$&$n=14$\\
    \hline
                 $ a_{n,1}$&      1& 1& 1& 1& 1& 1\\
 $ a_{n,2}$&77704& 365094& 1731796& 8279362& 39845688& 192852358\\
 $ a_{n,3}$&5079918& 41060566& 333328490& 2719109272& 22289691607&183567850172\\
 $a_{n,4}$&64387692& 772751250& 9167298700& 108088294433& 1270843912460&14929697027499\\
 $a_{n,5}$&316091052& 5408285876& 89123538198& 1433287497171& 22682092706596& 355124408138967\\
 $a_{n,6}$&754905976& 18475329110& 418529599512& 9007012080649&187105066615668& 3790795567806653\\
 $a_{n,7}$&942017540& 34352357886& 1087675765240& 31429263493327&852962648306500& 22135576491277369\\
 $a_{n,8}$&591756184& 35586878148& 1652106315312& 65646362702546&2352976034982156& 78449155249474142\\
 $a_{n,9}$&147939046& 19328939064& 1462267438608& 84177516803694&4099070898287721& 178300648622330310\\
 $a_{n,10}$&& 4295319792& 699493906520& 65014544794072&4544263507797230& 266042204038262334\\
 $a_{n,11}$&& & 139898781304& 27785057969284& 3112564949661888&259628599792491424\\
 $a_{n,12}$&& & & 5051828721688& 1202024862601332&159729470131220924\\
 $a_{n,13}$&&&&& 200337477100222&56262135084474878\\
 $a_{n,14}$&&&&&& 8655713089919212\\        
 \hline
    \end{tabular}
    \caption{Coefficients $a_{n,k}$ of the polynomials $p_9(w), \ldots , p_{14}(w).$}
    \label{tab:coeffs}
\end{table}
\normalsize

\subsection{Alternative representation of the polynomials.}
An alternative, and arguably more appealing -- because symmetries are manifestly obvious -- representation of the above polynomials (in either representation) follows if, instead of the polynomials, we consider the generating functions of the coefficients produced by the polynomials.  Let us write the generating function for LEGO configurations as: $$ L(x,w)=\sum_{n \ge 1} l_n(w) x^n,$$ where $l_n(w)=\sum_k d_{n,k} w^k $ is the generating function for LEGO structures of $n$ blocks.
Then we find
%\begin{tiny}
\begin{equation}
l_1(w) = \frac{1}{1-w},
\end{equation}
\begin{equation}
l_2(w)=\frac{1+w}{(1-w)^2}
\end{equation}
\begin{equation}
l_3(w):=\frac{1+8w+w^2}{(1-w)^3}
\end{equation}
\begin{equation}
l_4(w)=\frac{1+40w+40w^2+w^3}{(1-w)^4}=\frac{(1+w)(1+39w+w^2)}{(1-w)^4}
\end{equation}
\begin{equation}
l_5(w)=\frac{1+181w+586w^2+181w^3+w^4}{(1-w)^5}
\end{equation}
\begin{equation}
l_6(w)=\frac{w^5 + 808w^4 + 6321w^3 + 6321w^2 + 808w + 1}{(1-w)^6} = \frac{(1+w)(w^4 + 807w^3 + 5514w^2 + 807w + 1)}{(1-w)^6}
\end{equation}
\begin{equation}
l_7(w)=\frac{w^6 + 3649w^5 + 59786w^4 + 136856w^3 + 59786w^2 + 3649w + 1}{(1-w)^7}
\end{equation}
\begin{multline}
l_8(w)=\frac{w^7 + 16723w^6 + 529719w^5 + 2353319w^4 + 2353319w^3 + 529719w^2 + 16723w + 1}{(1-w)^8}\\
=\frac{(1+w)(w^6 + 16722w^5 + 512997w^4 + 1840322w^3 + 512997w^2 + 16722w + 1)}{(1-w)^8}
\end{multline}
\begin{equation}
\small
l_9(w)=\frac{w^8 + 77696w^7 + 4536018w^6 + 35539912w^5 + 67631792w^4 + 35539912w^3 + 4536018w^2 + \cdots +1}{(1-w)^9}
\end{equation}
\begin{equation}
l_{10}(w)=\frac{(w + 1)(w^8 + 365084w^7 + 37774766w^6 + 457775070w^5 + 1155830054w^4 + \cdots +1 )}{(1-w)^{10}}
\end{equation}
\begin{equation}
\footnotesize
l_{11}(w)=\frac{w^{10} + 1731786w^9 + 317742371w^8 + 6563015316w^7 + 34140174364w^6 + 57853453628w^5 + \cdots +1}{(1-w)^{11}}
\end{equation}
\begin{footnotesize}
\begin{equation}
l_{12}(w)=\frac{(w + 1)(w^{10} + 8279350w^9 + 2628036357w^8 + 81360845753w^7 + 584114706636w^6
+ 1189690624650w^5 + .. + 1)}{(1-w)^{12}}
\end{equation}
\end{footnotesize}
\begin{tiny}
\begin{equation}
    l_{13}(w)=\frac{w^{12}+39845676 w^{11}+21851389105 w^{10}+1050138509010 w^9+12240959078746 w^8+48737091894868 w^7+76237315665410 w^6+ \cdots}{(1-w)^{13}}
\end{equation}
%\end{scriptsize}
%\begin{tiny}
\begin{equation}
\frac{l_{14}(w)}{1+w}=\frac{w^{12} + 192852344w^{11} + 181060769610w^{10} + 12742118161339w^9 + 203139123944053w^8 + 1033179903486895w^7 + 1829371746531122w^6 + \cdots}{(1-w)^{14}}
\end{equation}
\end{tiny}

Clearly, for $p_9$ to $p_{14}$ we have utilised the symmetry in the numerator to write the numerator in a more compact form.
We observe that $$l_n(w)=(-1)^n\frac{1}{w}l_n \left ( \frac{1}{w} \right ),\,\, w \ne 0, \,\, l_n(0) =1,$$
and that $$l_n(w)=\frac{A_n(w)}{(1-w)^n},$$ where $A_n(w)$ is a symmetric, unimodal polynomial of degree $n-1.$ Further, when $n$ is even, $A_n=(1+w)B_n(w),$ where $B_n(w)$ is again a symmetric, unimodal polynomial, but of degree $n-2.$

\subsection{Further observations about the polynomial coefficients.}
In conventional polynomial notation, one has: $$p_n(w) = \sum_{i=0}^{n-1} b_{n,i}w^i$$

We can readily transfer between the conventional and binomial representations as $$b_{n,i} = \sum_{k=i}^{n-1} \frac{{\mathcal S}_{k+1}^{(i+1)}}{k!}a_{n,k},$$ where ${\mathcal S}_{k}^{(i)}$ are Stirling numbers of the first kind.
Equivalently,  $$a_{n,i} = (i-1)!\sum_{k=0}^{n-1} {\mathfrak S}_{k+1}^{(i)}b_{n,k},$$ where ${\mathfrak S}_{k}^{(i)}$ are Stirling numbers of the second kind.

%\;\;,\;\; f_n(w + \tfrac{1}{2}) = \sum_{i=0}^{n-1} b_{n,i}w^i$$
% Just by definition we have the relation $$a_{n,i} = \sum_{k=i}^{n-1}\left(-\frac{1}{2}\right)^{k-i}\binom{k}{i} b_{n,k} $$
Looking at the list \ref{sec:polylist} of polynomials we notice that the constant terms $b_{n,0}$ equals $1$ if $n$ is odd and $-1$ if $n$ is even. This can be proved true for all $n$ by theorem \ref{conj} because then $b_{n,0}=p_n(0) = p_n(1)(-1)^{n-1} = (-1)^{n-1}$.
\noindent
Since there is only a single LEGO structure with $w=1,$  it follows that $\sum_{i=0}^{n-1} b_{n,i}=1,$ since $p_n(1)=1.$

We also notice a simple relation between the leading coefficients $b_{n,n-1}$ and the coefficient of the second highest order term $b_{n,n-2}$: $$b_{n,n-2} = -\frac{n-1}{2}b_{n,n-1}$$ 
This can also be shown to follow from theorem \ref{conj} using the fact that the coefficient of the second highest order term of $p_n(w+\frac{1}{2})$ is then $0$.
\newline \newline \noindent

\subsection{Polynomial coefficient growth}
\begin{comment}
We have visualized the growth of the coefficients of the polynomials in figure \ref{fig:coeff1} and \ref{fig:coeff2} below. {\bf [NOTE: These plots are very similar to those of the pyramids. We could mention that. I examine this more in my thesis and prove for the pyramids that the growth constant of the polynomials leading coefficients as well as the growth of the growth constants is equal to $e$ (pp 50). Could be mentioned in relation to the analysis below].}
\newline

\begin{figure}
    \centering
    \includegraphics[width=1\linewidth]{coefficients_plot1_2024.pdf}
    \caption{Coefficients of $p_n$ for $n=1..14$ plotted against the order of their term. Dotted lines group coefficients of each individual polynomial. One solid line is drawn though the leading coefficients, one through the second highest order term coefficients and so on. The rightmost point is the leading coefficient of $p_{14}$.}
    \label{fig:coeff1}
\end{figure}
\begin{figure}
    \centering
    \includegraphics[width=1\linewidth]{coefficients_plot2_2024.png}
    \caption{Coefficients of $p_n$ for $n=1..14$ plotted against the degree of their polynomial. Dotted lines group coefficients of same order. Solid lines are drawn as in figure \ref{fig:coeff1}}
    \label{fig:coeff2}
\end{figure}

\end{comment}

Note that as $n$ gets large, the polynomials will be dominated by the leading-order term. Looking at how the leading-order term grows then tells us how the growth constant grows with $n.$ We have analysed the sequence given by the leading-order polynomial coefficients, and estimate that it grows as $3.5730^n/n.$ This implies that the growth constants should grow linearly, with the linear term being $3.5730n.$

It follows that as $n$ increases, the growth constants should increase by a multiplicative factor asymptotically approaching 3.5730. 

From Table \ref{tab:2} one sees that going from $9 \times 1$ to $10 \times 1$ tiles, the growth constant already increases by 3.577. Fitting the data from Table \ref{tab:2} to the linear function $c+3.5730n,$ we estimate $c\approx -1.780,$ so that
the growth constants can be estimated by $\mu(n) \approx 3.5730n-1.780,$ a result that should become increasingly accurate as $n$ increases.
\newline

\section{$2 \times 4$ structures}
We turn now to a three-dimensional tiling problem, which, historically, was the first such problem involving the enumeration of LEGO structures in the literature.
The history of this problem, and its definition, is given in the Introduction. The first 10 coefficients are given in the OEIS as sequence A112389, where the coefficients at order 9 and 10 were found by Johan Nilsson and Matthias Simon respectively.

We have attempted to extend the sequence approximately, as discussed above for the $w \times 1$ structures. Unfortunately, we cannot do nearly as well, as these series are a bit too short to construct many differential approximants. Nevertheless, we can make some minor progress. Assuming we have only 9 coefficients, we predict the 10th as $8.2058\times 10^{16} \pm 2.3 \times 10^{12},$ which agrees with the exact value 
$8.2057965\ldots \times 10^{16}$ within the quoted error, which is one standard deviation. 

Using all 10 terms, we predict the next three terms in this sequence to be $8.3649 \times 10^{18} \pm 2.4 \times 10^{14},$ $8.6304 \times 10^{20} \pm 7.0 \times 10^{16},$ and $8.994 \times 10^{22} \pm 2.1 \times 10^{19}.$ The errors quoted are 1 standard deviation, so for full confidence as bounds, they should probably be tripled.

We have analysed this sequence in a manner similar to the preceding analyses, using two further approximate coefficients.

We estimate the growth constant to be $\mu = 117.25 \pm 0.05.$ This is in excellent agreement with the Monte Carlo estimate
of Eilers \cite{E16, DE05} of $\mu \approx 117.$ Both these estimates are consistent with the rather wide bounds found by Eilers, $81 < \mu < 177.$ 

Assuming log-convexity, which is well-supported by the data, one obtains the (non-rigorous) lower bound $\mu >100.47.$ 
For the critical exponent we find $g \approx -3/2,$ and conjecture equality, as tentatively observed previously by Abrahamsen and Eilers \cite{AE11}.

So we estimate that the coefficients behave as $$a_n \sim A\mu^n \cdot n^{-3/2},$$
where $\mu=117.25 \pm 0.05,$ and $A \approx 0.00551.$ It follows that the generating function behaves as 
$$A(x)=\sum_n a_n x^n \sim C\sqrt{1-\mu x},$$ where $C = A\cdot \Gamma(-1/2)=-2\sqrt{\pi}A \approx -0.0195.$
\section{Conclusion}

We have studied several different LEGO enumeration problems and conjecture that, in every case, the number of such objects of size $n$ tiles grows as do most such enumeration problems, with exponential growth and a sub-dominant power law term. More precisely, the coefficients behave as $$a_n \sim A \mu^n n^g,$$ and in the tables above we summarise our estimates of the critical parameters.
As is frequently the case, we find that there is an element of universality, in that the critical exponent appears to depend only on the dimension of the problem. So the various planar problems all have exponent $g=-1,$ while the three-dimensional problem has exponent $g=-3/2.$

We have given non-rigorous lower bounds for all the growth constants, based on the unproven assumption of log-convexity of the coefficients. Note that a stronger condition than log-convexity is that the coefficients can be expressed as the moments of a Stieltjes series. In such cases one has more powerful methods to find the lower bounds. However we have tested the available data and found that, after a sufficient number of terms, the Hankel determinants become negative, thus ruling out the possibility that these sequences are Stieltjes moment sequences.

As well as analysing the sequences that arise when fixing the tile size and varying the number of tiles, we have also constructed the polynomials that arise when we fix the number of tiles $n$ and vary the tile size instead. We have proved that this polynomial structure persists for arbitrary tile size.

In this way we were able to construct polynomials for $n \le 14.$ Note that the knowledge of such polynomials allows one to generate the first 14 terms in the series for {\em any} tile size. 
\section{Acknowledgements}
We would like to thank Andrew Conway, who was involved in very helpful discussions and provision of some test data, and also a computer program that drew some invaluable pictures.  Michal Adamaszek first informed us that the existence of polynomial structures for all tile sizes can be proved by Ehrhart theory, and that the symmetry relation of theorem \ref{conj} can be similarly proved. RMN would like to thank his thesis supervisors S{\o}ren Eilers and Michal Adamaszek for their inspiring supervision of his Master's thesis.

\section{Appendix A}
In Tables \ref{tab:1a}, \ref{tab:1b}, and \ref{tab:1c}  we give the coefficients generated for $w \times 1$ LEGO structures for $1 < w < 11.$  Note that the solution for $w=1$ is trivial, as every coefficient is 1, so that the generating function is $\frac{1}{1-x}.$
\begin{table}
\begin{tabular}{ll}
\hline
    $2 \times 1$ & $3 \times 1$\\
    \hline
     1    & 1  \\
      3   &  5 \\
       11  & 31  \\
         44&210   \\
  186       &1506  \\
  814       &11190  \\
  3656      &85357  \\
16731      & 663539 \\
 77705 & 5235327 \\
 365095 & 41790755 \\
 1731797  & 336792083 \\
 8279363  & 2735667997 \\
 39845689  & 22369382984  \\
 192852359  & 183953554889\\
 937986507  & 1520171511036\\
 4581678031  & 12616393722193 \\
 22464030959  & 105102766371328 \\
 110509938701   & 878505513933208 \\
 545269104263  & 7364975984121163  \\
 2697646445713  & 61910345691875194 \\
 13378627003520  & 521684370818376093 \\
 66495716465315  &  \\
 331167284581601  &  \\
 1652340114446553  &  \\
 8258197397705302  &  \\
 41337852343827210  &  \\
 207222462319935608  &  \\
 1040176220193951923  &  \\
 5227785863956950802 &  \\
    \end{tabular}%
    \begin{tabular}{ l @{} }
    \hline
    $4 \times 1$ \\
    \hline
    1\\
 7\\
 61\\
 581\\
 5861\\
 61271\\
 657614\\
 7193219\\
 79860559\\
 897028231\\
 10172479559\\
 116270460336\\
 1337832524346\\
 15480979135090\\
 180022037747385\\
 2102381303716934\\
 24645280048119407\\
 289872904035637011\\
 \\
 \\
 \\
 \\
 \\
 \\
 \\
 \\
 \\
 \\
 \\
\end{tabular}%
   \begin{tabular}{ l @{} }
    \hline
    $5 \times 1$ \\
    \hline
1\\
 9\\
 101\\
 1239\\
 16101\\
 216849\\
 2998512\\
 42256845\\
 604432145\\
 8747114649\\
 127799631123\\
 1881988447984\\
 27899365888831\\
 415945374759428\\
 6231734780546397\\
 93764902478331859\\
\\
\\
 \\
 \\
 \\
 \\
 \\
 \\
 \\
 \\
 \\
 \\
 \\
\end{tabular}
 \caption{Counts, $w= 2 \cdots 5$.}
    \label{tab:1a}
\end{table}
\begin{table}
\begin{tabular}{l}
\hline
    $6 \times 1$ \\
    \hline
1\\
 11\\
 151\\
 2266\\
 36026\\
 593626\\
 10042895\\
 173161417\\
 3030425857\\
 53656702121\\
 959162221383\\
 17281565000365\\
 313447065417759\\
 5717551221539989\\
 104806127257244312\\
    \end{tabular}%
    \begin{tabular}{ l @{} }
    \hline
    $ 7 \times 1$ \\
    \hline
    1\\
 13\\
 211\\
 3744\\
 70386\\
 1371474\\
 27437278\\
 559425841\\
 11577238011\\
 242401744741\\
 5124062727409\\
 109173250638699\\
 2341575901296882\\
 50508564635618501\\
\\
\end{tabular}%
   \begin{tabular}{ l @{} }
    \hline
    $8 \times 1$ \\
    \hline
1\\
 15\\
 281\\
 5755\\
 124881\\
 2808695\\
 64858487\\
 1526434767\\
 36462882707\\
 881236421615\\
 21502149574037\\
 528803772805851\\
 13091742116167208\\
 325960647568001867\\
\\
\end{tabular}%
 \begin{tabular}{ l @{} }
    \hline
    $9 \times 1$ \\
    \hline
1\\
 17\\
 361\\
 8381\\
 206161\\
 5256281\\
 137596027\\
 3670989809\\
 99408158023\\
 2723515113137\\
 75333060413645\\
 2100219743338293\\
 58943255239894358\\
 1663674497216953661\\
\\
\end{tabular}
 \caption{Counts, $w= 6 \cdots 9$.}
    \label{tab:1b}
\end{table}

\begin{table}
\begin{tabular}{l}
\hline
    $10 \times 1$  \\
    \hline
1\\
 19\\
 451\\
 11704\\
 321826\\
 9172174\\
 268398178\\
 8004557671\\
 242301915439\\
 7420709506579\\
 229446856339449\\
 7150594512744185\\
 224332637060251284\\
 7077952677585338683\\

\end{tabular}%

 \caption{Counts, $w=10.$ }
    \label{tab:1c}
\end{table}
 \newpage
\section{Appendix B}
\subsection{Polynomials for interlocking $n \times 1$ structures.}\label{sec:polylist}

\begin{tiny}
\begin{equation}
1
\end{equation}
\begin{equation}
2 w-1
\end{equation}
\begin{equation}
5 w^{2}-5 w+1
\end{equation}
\begin{equation}
\frac{41}{3} w^{3}-\frac{41}{2} w^{2}+\frac{53}{6} w -1
\end{equation}
\begin{equation}
\frac{475}{12} w^{4}-\frac{475}{6} w^{3}+\frac{635}{12} w^{2}-\frac{40}{3} w+1
\end{equation}
\begin{equation}
\frac{713}{6} w^{5}-\frac{3565}{12} w^{4}+\frac{811}{3} w^{3}-\frac{1301}{12} w^{2}+\frac{55}{3} w-1
\end{equation}
\begin{equation}
\frac{16483}{45} w^{6}-\frac{16483}{15} w^{5}+\frac{91309}{72} w^{4}-\frac{8459}{12} w^{3}+\frac{69491}{360} w^{2}-\frac{1423}{60} w+1
\end{equation}
\begin{equation}
\frac{1449881}{1260} w^{7}-\frac{1449881}{360} w^{6}+\frac{2029483}{360} w^{5}-\frac{289801}{72} w^{4}+\frac{557479}{360} w^{3}-\frac{56077}{180} w^{2}+\frac{3076}{105} w-1
\end{equation}
\begin{equation}
\frac{73969523}{20160} w^{8}-\frac{73969523}{5040} w^{7}+\frac{6959401}{288} w^{6}-\frac{7605373}{360} w^{5}+\frac{30387061}{2880} w^{4}-\frac{2170541}{720} w^{3}+\frac{470965}{1008} w^{2}-\frac{14629}{420} w+1
\end{equation}
\begin{equation}
\frac{89485829}{7560} w^{9}-\frac{89485829}{1680} w^{8}+\frac{25421659}{252} w^{7}-\frac{2090137}{20} w^{6}+\frac{2899702}{45} w^{5}
-\frac{5787193}{240} w^{4}+\frac{8094845}{1512} w^{3}-\frac{46233}{70} w^{2}+\frac{842}{21} w-1
\end{equation}
\begin{multline}
\frac{17487347663}{453600} w^{10}-\frac{17487347663}{90720} w^{9}+\frac{1248502385}{3024} w^{8}-\frac{7482700037}{15120} w^{7}+\frac{7861487219}{21600} w^{6}\\
-\frac{731661827}{4320} w^{5}+\frac{225142945}{4536} w^{4}-\frac{200439697}{22680} w^{3}+\frac{5598629}{6300} w^{2}-\frac{5641}{126} w+1
\end{multline}
\begin{multline}
\frac{90211227173}{712800} w^{11}-\frac{90211227173}{129600} w^{10}+\frac{60372859255}{36288} w^{9}-\frac{137057149307}{60480} w^{8}+\frac{146560866553}{75600} w^{7}\\
-\frac{46757432153}{43200} w^{6}+\frac{10290244517}{25920} w^{5}-\frac{2435385197}{25920} w^{4}+\frac{3112726853}{226800} w^{3}-\frac{28814141}{25200} w^{2}+\frac{167581}{3465} w-1
\end{multline}
\begin{multline}
\frac{1300892708443}{3110400} w^{12}-\frac{1300892708443}{518400} w^{11}+\frac{20596152314399}{3110400} w^{10}-\frac{1047722086921}{103680} w^{9}\\
+\frac{71609141325941}{7257600} w^{8}-\frac{7788460982861}{1209600} w^{7}+\frac{8860861426997}{3110400} w^{6}-\frac{17603330867}{20736} w^{5}\\
+\frac{64511466371}{388800} w^{4}-\frac{2619494047}{129600} w^{3}+\frac{213101387}{151200} w^{2}-\frac{126821}{2520} w+1
\end{multline}
\begin{multline}
\frac{2163928272479803}{1556755200} w^{13}-\frac{2163928272479803}{239500800} w^{12}+\frac{316983608777}{12150} w^{11}-\frac{960262175626309}{21772800} w^{10}\\+
\frac{43931020425919}{907200} w^{9}-\frac{262825432507213}{7257600} w^{8}+\frac{25598928894827}{1360800} w^{7}-\frac{147962647433887}{21772800} w^{6}\\
+\frac{2615911547453}{1555200} w^{5}-\frac{1502460928697}{5443200} w^{4}+\frac{94320295367}{3326400} w^{3}-\frac{690785069}{415800} w^{2}+\frac{6511}{130} w-1
\end{multline}
\end{tiny}


\newpage
\begin{thebibliography}{10}



\bibitem{AE11} M Abrahamsen and S Eilers, {\em Experimental Mathematics}, {\bf 20} (2) 144-152, (2011).

\bibitem{BR} M Beck and S Robins, {\em Computing the continuous discretely}, (Springer, 2nd ed. 2015), Thm 4.1.

\bibitem{C74}  J K Christiansen, {\em Taljonglering med klosder-eller talrige klodser}, Klodshans (LEGO company newsletter), (1974).

\bibitem{DE05} B Durhuus and S Eilers, {\em On the entropy of LEGO}, arXiv:/math/0504039, (2005), J. Appl. Math. Comput. {\bf 45} 433-448, (2014).

\bibitem{E16} S Eilers, {\em The LEGO counting problem }, The American Mathematical monthly, {\bf 123} (5) 415--426 (2016).

\bibitem{F23} M. Fekete, {\em \"Uber die Verteilung der Wurzeln bei gewissen algebraischen
Gleichungen mit ganzzahligen Koeffizienten}, Math. Zeit. {\bf 17}
 228--249  (1923).
\bibitem{G09} A J Guttmann,
  \emph{Polygons, Polyominoes and Polycubes} Lecture Notes in Physics {\bf 775}, ed. A J Guttmann, Springer, (Heidelberg), (2009).

\bibitem{G16} A J Guttmann,
{\em Series extension: predicting approximate series coefficients from a finite number of exact coefficients},
J. Phys. A: Math. Theor. {\bf 49}  415002 (27pp), (2016).

\bibitem{EJ09} Ian G. Enting and Iwan Jensen,
{\em Exact Enumerations},
Polygons, Polyominoes and Polycubes, ed. A J Guttmann, Springer (143-180pp) (2009).

\bibitem{DE10}
  B Durhuus and S Eilers,
  {\em Enumeration of pyramids of one-dimensional pieces of arbitrary fixed integer length},
  Discrete Mathematics and Computer Science proc. (145-160pp) (2010)

 \bibitem{N16} R M Nilsson, {\em On the number of flat LEGO structures}, Master's thesis in Mathematics, University of Copenhagen, (2016).

 \bibitem{OEIS}  {\em The On-Line Encyclopaedia of Integer Sequences}, OEIS Foundation Inc. (2014), https://oeis.org
\bibitem{S74} R P Stanley, {\em Combinatorial reciprocity theorems}, Adv. Math., {\bf 14} 194-253, (1974).
\end{thebibliography}
\end{document}